\documentclass[dvipsnames,reqno]{amsart}
\usepackage{fullpage}
\usepackage{todonotes}
\usepackage[all,2cell]{xy}
\usepackage{graphicx, amsmath, amssymb, amsthm}
\usepackage{newclude}
\usepackage{mathrsfs}
\usepackage[hidelinks,pagebackref=true]{hyperref}
\usepackage{sseq}
\usepackage{backref}
\usepackage{eucal}
\usepackage{multirow,bigdelim}
\usepackage{array}

\let\tops=\texorpdfstring

\newcommand{\cF}{\mathcal F}

\newcommand{\TPhi}{\operatorname{T\Phi}}

\newcommand{\BB}{\mathcal A}

\newcommand{\<}{\langle}
\let\>=\undefined
\newcommand{\>}{\rangle}

\newcommand{\lf}{\left\lfloor}
\newcommand{\rf}{\right\rfloor}

\newcommand{\defeq}{\mathrel{:=}}
\newcommand{\tnsr}{\otimes}

\newcommand{\iso}{\morph^\sim}

\newcommand{\bdry}{\partial}

\newcommand{\psr}[1]{[\![#1]\!]}

\DeclareMathOperator{\Tor}{Tor}

\newcommand{\morph}{\mathop{\longrightarrow}\limits}

\newcommand{\xra}{\xrightarrow}

\renewcommand{\projlim}{\mathop{\operatorname*{lim}\limits_{\longleftarrow}}\limits}
\renewcommand{\lim}{\projlim}

\DeclareMathOperator{\tr}{tr}

\newcommand{\res}{\mathrm{res}}

\newcommand{\C}{\mathbf C}
\newcommand{\F}{\mathbf F}
\newcommand{\N}{\mathbf N}
\renewcommand{\O}{\mathcal O}

\let\sec=\S
\renewcommand{\S}{\mathbf S}

\newcommand{\Z}{\mathbf Z}

\newcommand{\K}{{\mathrm{K}}}

\newcommand{\THH}{\mathrm{THH}}
\newcommand{\TR}{\mathrm{TR}}
\newcommand{\TC}{\mathrm{TC}}
\newcommand{\TCmin}{\mathrm{TC}^-}
\newcommand{\TP}{\mathrm{TP}}
\newcommand{\TF}{\mathrm{TF}}

\newcommand{\T}{\mathbb T}

\newcommand{\Ainf}{{\mathbf{A}_{\mathrm{inf}}}}

\usepackage{relsize}
\usepackage[bbgreekl]{mathbbol}
\DeclareSymbolFontAlphabet{\mathbb}{AMSb} %to ensure that the meaning of \mathbb does not change
\DeclareSymbolFontAlphabet{\mathbbl}{bbold}

\newcommand{\prism}{{\mathlarger{\mathbbl{\Delta}}}} 

\newcommand{\RO}{\mathrm{RO}}
\newcommand{\RU}{\mathrm{RU}}
\newcommand{\JO}{\mathrm{JO}}
\newcommand{\Pic}{\operatorname{Pic}}
\newcommand{\m}{\underline}
\newcommand{\mpi}{\underline\pi}
\newcommand{\rog}{\bigstar}

\newcommand{\Nyg}{{\mathcal N}}

\newcommand{\mup}[1]{\ar@/_1em/[u]_-{#1}}
\newcommand{\mdown}[1]{\ar@/_1em/[d]_-{#1}}
\newcommand{\mcup}[2][]{\ar@[#1]@/_1em/[u]_-{\color{#1}#2}}
\newcommand{\mcdown}[2][]{\ar@[#1]@/_1em/[d]_-{\color{#1}#2}}

\newcommand{\Lewis}[4]{\xymatrix{
  #1 \mdown{#3}\\#2 \mup{#4}
}}

\newcommand{\sLewiss}[7]{\xymatrix@=1.8em{
  #1 \mdown{#4}\\#2 \mup{#5} \mdown{#6} \\#3 \mup{#7}
}}

\newdir{ >}{{}*!/-10pt/@{>}}
\newcommand{\pullback}{\ar@{}[dr]|<<{\mbox{\Huge$\lrcorner$}}}
\newcommand{\pushout}{\ar@<2pt>@{}[ul]|<{\mbox{\Huge$\ulcorner$}}}

\newcommand{\adjnctn}[4]{\xymatrix@1{
  #1 \ar@<1ex>[r]^-{#3} \ar@{}[r]|-{\bot} & #2 \ar@<1ex>[l]^-{#4}
}}
\newcommand{\dadjnctn}[6][2em]{
  \xymatrix@1@C=#1{
    #2 \ar@<2ex>[r]^-{#4} \ar@<-2ex>[r]_-{#6}
    \ar@{}@<1.2ex>[r]|-{\scriptscriptstyle\bot} \ar@{}@<-1.2ex>[r]|-{\scriptscriptstyle\bot}
    &
    #3 \ar[l]|-{#5}
  }
}

\newcommand{\dimorph}[4]{\xymatrix@1{
  #1 \ar@<1ex>[r]^-{#3} \ar@<-1ex>[r]_-{#4} & #2
}}
\newcommand{\twoCell}[5]{\xymatrix@1{
  #1 \ar@<1ex>[r]^-{#3} \ar@{}[r]|{\Downarrow #5} \ar@<-1ex>[r]_-{#4} & #2
}}

\newcommand{\ndo}[2]{\xymatrix@1{
  #1 \ar[r]^-{#2} & #1
}}

\newtheorem{theorem}{Theorem}[section]
\newtheorem*{theorem*}{Theorem}

\newtheorem{proposition}[theorem]{Proposition}

\theoremstyle{definition}

\newtheorem*{definition*}{Definition}

\newtheorem{remark}[theorem]{Remark}
\newtheorem*{remark*}{Remark}

\newtheorem{example}[theorem]{Example}
\newtheorem*{example*}{Example}

\newtheorem*{heuristic*}{Heuristic}

\newcommand{\MWitt}{
 \begin{tikzpicture}[x=1.2ex,y=1.2ex]
    \draw (0,0) rectangle +(1.,1);
 \end{tikzpicture}
}

\newcommand{\Mdual}{
 \protect{\begin{tikzpicture}[x=1.2ex,y=1.2ex]
    \draw (0,0) rectangle +(1.,1);
    \fill (0.5,0) -- (.5,1) -- (1,1) -- (1, 0) -- cycle;
 \end{tikzpicture}}
}

\newcommand{\Mtors}{\bullet}

\makeatletter
\providecommand\@dotsep{5}
\renewcommand{\listoftodos}[1][\@todonotes@todolistname]{%
  \@starttoc{tdo}{#1}}
\makeatother

\SelectTips{cm}{}
\UseAllTwocells
\NoCompileMatrices

\begin{document}

\title{\tops{$\RO(\T)$}{RO(T)}-graded $\TF$ of perfectoid rings}

\author[Y.~J.~F.~Sulyma]{Yuri~J.~F. Sulyma}
\address{Brown University \\ Providence, RI 02912}
\email{yuri\_sulyma@brown.edu}
\thanks{Sulyma was supported in part by NSF grants DMS-1564289 and DMS-1151577}

\begin{abstract}
For a perfectoid ring $R$, we compute the full $\RO(\T)$-graded ring $\mathrm{TF}_\bigstar(R;\Z_p)$. This extends and simplifies work of Gerhardt and Angeltveit-Gerhardt. In even degrees, we find an $\RU(\mathbb T)$-graded version of B\"okstedt periodicity, with some additional classes in the case of perfect $\F_p$-algebras. In odd degrees, we find extremely intricate and rather mysterious torsion. We also discuss the $\RO(\T)$-graded homotopy Tambara functors $\mpi_\rog\THH(R;\Z_p)$.
\end{abstract}

\maketitle
\tableofcontents

% \listoftodos

% \include{orientations}
% !TEX root=./rog.tex
\section{Introduction}

Topological Hochschild homology and its variants carry significant arithmetic information in their homotopy groups. For instance, Hesselholt \cite{LarsPTyp} has shown that $\TR^n_*$ realizes the de Rham-Witt complex: for $S$ a smooth algebra over a perfect $\F_p$-algebra $k$,
\[ \TR^n_*(S) = \Tor^{W_n(k)}_0(W_n\Omega^*_{S/k}, W_n(k)[\sigma]), \qquad |\sigma|=2. \]
The framework of genuine equivariant homotopy theory allows one to consider homotopy groups graded not only by integers $*\in\Z$, but by arbitrary virtual representations $\rog\in\RO(\T)$ of the circle group. For virtual representations of the form $\rog=*-V$ with $*\in\Z$ and $V$ an \emph{actual} $\T$-representation, the groups $\TR^n_\rog(k)$ arise when computing $\K$-theory of monoid algebras \cite[\sec9]{HMFinite}, and have been applied to many such calculations \cite{HMCyclicPolytopes,SpeirsTrunc,HesselholtAxes,SpeirsAxes}.
% \textcolor{red}{how many papers do I have to cite here? It's a small field but there's still like 15 different papers that do this}.
More recently, the degrees $\rog=*+V$, for $V$ an \emph{irreducible actual} $\T$-representation, were the key calculational input in identifying the slice filtration on $\THH$ \cite{SulSliceTHH}.

These degrees are a tiny portion of the representation ring $\RO(\T)$, and the groups in this range are fairly easy to compute. Although no such applications are known for general $\rog$, these groups are quite interesting in their own right, and should capture deeper structure of the de Rham-Witt complex. They are also significantly more difficult to compute than the special ranges indicated above.

Gerhardt \cite{TeenaTR} and Angeltveit-Gerhardt \cite{ROS1TR} have given algorithms to compute the \emph{groups} $\TR^n_\rog(k)$ for arbitrary $\rog$. In this paper, we revisit their work in order to compute the $\RO(\T)$-graded \emph{ring} $\TF_\rog(R;\Z_p)$, for $R$ either a $p$-torsionfree perfectoid ring or a perfect $\F_p$-algebra.

Their method is a novel spectral sequence they call the ``homotopy orbits to TR spectral sequence'' (HOTRSS). This is a slick way of packaging the usual strategy of induction on the isotropy separation sequence. We will present several upgrades to the HOTRSS:

\begin{enumerate}
\item We name classes using the ``gold elements'' $a_{\lambda_i},\,u_{\lambda_j}$ of equivariant homotopy theory \cite[\sec3]{HHR_KR}. While this is in some sense just a change of notation, it makes the calculations considerably more transparent, and makes it easy to track the multiplicative structure and solve extension problems.

\item Using the ``$q$-gold relation'' \cite[Lemma 4.9]{SulSliceTHH}, we can extend the method from perfect $\F_p$-algebras to perfectoid rings.

\item We give a ``homotopy orbits to TF spectral sequence'' (HOTFSS) for calculating $\TF_\rog(R;\Z_p)$ directly, starting from $\TC^-_*(R;\Z_p)$ and $\TP_*(R;\Z_p)$. The HOTFSS has much less torsion than the HOTRSS, and we are able to obtain a closed form for $\TF_\rog(R;\Z_p)$, the even part of which is simple and conceptual. If desired, the finite-level $\TR^{n+1}_\rog$ groups can be deduced using the long exact sequence
\[ \TF_{\rog+\lambda_n} \xra{a_{\lambda_n}} \TF_\rog(R;\Z_p) \to \TR^{n+1}_\rog(R;\Z_p) \xra\partial \TF_{\rog+\lambda_n-1} \]

\item Alternatively, we show how to formulate the HOTRSS as a spectral sequence of \emph{Mackey functors}, which we call the HO\m{TR}SS. Combined with our naming scheme, one can in principle determine the full $\RO(\T)$-graded Tambara functor $\mpi_\rog\THH(R;\Z_p)$. (The action of the Norm on the gold elements is given in \cite[Lemma 3.13]{HHR}.) However, we have not attempted to give a closed form for this.
\end{enumerate}
These advances should make it easier to import the result into algebraic geometry.

Let us review the setup; a full treatment of the necessary background can be found in \cite[Part 1]{SulSliceTHH}. The complex representation ring of $\T$ is
\[ \RU(\T) = \Z[\lambda^{\pm1}], \]
where $\lambda^i$ is the irreducible $\T$-representation in which $z\in\T$ acts via multiplication by $z^i$. As real representations, $\lambda^i$ and $\lambda^{-i}$ are isomorphic via complex conjugation. In the $p$-local setting we have $S^{\lambda^i} \simeq S^{\lambda^j}$ iff $v_p(i)=v_p(j)$, so we set $\lambda_i=\lambda^{p^i}$ and $\lambda_\infty=\lambda^0$. The real representations of $\T$ are all of the form $\alpha$ or $\alpha-1$ for a complex representation $\alpha$, and we will always write them this way.

\begin{remark}
Since we are really grading over $\bigoplus_{i\in\N\cup\{\infty\}}\Z\<\lambda_i\>$, it would be more accurate to speak of $\JO(\T)$-grading, or $\Pic$-grading; in particular, the multiplicative structure of $\RO(\T)$ is irrelevant in this paper. On the other hand, one of the main observations of \cite{SulSliceTHH} is that the ring structure of $\RO(\T)$---or at least the ability to form $q$-analogues---gives a clean way to keep track of induced representations.
\end{remark}

Given $\alpha\in\RU(\T)$, we define
\begin{align*}
  \alpha^{(r)} &= \rho_{p^r}^*(\alpha^{C_{p^r}})\\
  d_r(\alpha) &= \dim_\C(\alpha^{(r)}),
\end{align*}
where $\rho_{p^r}\colon\T\iso\T/C_{p^r}$ is the root isomorphism. Explicitly, given a virtual $\T$-representation
\begin{equation}
  \label{eq:irred-decomp}
  \alpha = k_0 \lambda_0 + \dotsb + k_n \lambda_n + k_\infty \lambda_\infty,
\end{equation}
we have
\begin{align*}
  \alpha^{(r)} &= k_r \lambda_0 + \dotsb + k_n \lambda_{n-r} + k_\infty \lambda_\infty\\
  d_r(\alpha) &= k_\infty+\sum_{i=r}^{n} k_i.
\end{align*}

A virtual representation $\alpha\in\RU(\T)$ can be represented in two ways: via the irreducible decomposition as in \eqref{eq:irred-decomp}, or via the dimension-sequence $d_\bullet(\alpha)$. The former is more natural for computing with cell structures, as was done in \cite{SulSliceTHH}; the latter is more natural when using the isotropy separation sequence, which is what do here. Therefore, we will always write such a representation as $\alpha=(d_0,\dotsc,d_n;d_\infty)$. Note that this encoding distinguishes between $\alpha$ and $\alpha+0\lambda_{n+1}$; in fact these will give different spectral sequences, although the final answer is of course the same.

For $i\ge0$, set
\begin{align*}
  a_i &= a_{\lambda_{i-1}}^{-1} a_{\lambda_i}\\
  u_i &= u_{\lambda_{i-1}} u_{\lambda_i}^{-1}
\end{align*}
where $a_{\lambda_{-1}}\defeq1$ and $u_{\lambda_{-1}}\in\TF_2(R;\Z_p)$ is defined to be the B\"okstedt generator. Note that $a_i\in\pi^\T_\rog\S$ by \cite[Corollary 2.9]{HZhZ}, and $u_i\in\TF_\rog(R;\Z_p)$ by Tsalidis' theorem \cite[Theorem 5.1]{ROS1TR}. These classes are well-suited to the dimension-sequence encoding since they only change $d_i$ by $\pm1$. With this notation, the $q$-gold relation \cite[Lemma 4.9]{SulSliceTHH} is
\[ a_i u_i = \phi^i(\xi). \]

The key technique in this paper is keeping track of names of elements, and the elements $a_i,\,u_i$ make this easy. All the classes in this paper have the form
\[ g_0(\alpha) \dotsm g_n(\alpha) u_{\lambda_n}^{d_\infty} \]
where $g_i(\alpha)$ is either $u_i^{d_i(\alpha)}$ or $a_i^{-d_i(\alpha)}$; this is just a matter of convincing oneself the terms have the correct degree. In particular, we define for convenience
\[ \vartheta_r^\alpha \defeq a_0^{-d_0} \dotsm a_r^{-d_r} u_{r+1}^{d_{r+1}}\dotsm u_n^{d_n} u_{\lambda_n}^{d_\infty}. \]
Now set
\[
  \BB \defeq A[a_i, u_j, u_{\lambda_i} \mid i\ge 0, j\ge 1]
\]
where $A=\Ainf(R)=\TF_0(R;\Z_p)$. This is our ``basic block'' \cite[Remark 2.2(ii)]{Greenlees4Real} or ``ground ring''.

\begin{theorem}
\label{thm:rog-trans}
Let $(A,(\xi))$ be a transversal perfect prism, corresponding to a $p$-torsionfree perfectoid ring $R=A/\xi$. The $\RU(\T)$-graded $\TF$ of $R$ is
\begin{align*}
  \TF_\rog(R;\Z_p) &= \BB[u_0].
\end{align*}

The remaining $\RO(\T)$-graded groups are as follows: let $\alpha=(d_0,\dotsc,d_n;d_\infty)$. Then $\TF_{\alpha-1}(R;\Z_p)$ has one summand for each contiguous subsequence $(d_s,d_{s+1},\dotsc,d_r)$ of $\alpha$ such that
\begin{itemize}
  \item $d_i\ge0$ for all $s\le i\le r$;

  \item either $s=0$, or $d_{s-1}<0$;

  \item if $d_\infty<0$, then either $r=n$ or $d_{r+1}<0$;

  \item if $d_\infty\ge0$, then $r<n$ and $d_{r+1}<0$.
\end{itemize}
Given such a subsequence, let
\begin{align*}
  o^\alpha_{r,s} &= \prod_{i=s}^r \phi^i(\xi)^{d_i}\\
  f^\alpha_r &= \prod_{i>r,\,d_i<0} \phi^{i}(\xi)^{-d_i}
\end{align*}
The corresponding summand is
\[\left\{\begin{array}{rl}
  A/o^\alpha_{r,s}\<\Sigma^{-1} \vartheta_r^\alpha\> & d_\infty<0\\
  A/(o^\alpha_{r,s},f^\alpha_r)\<\Sigma^{-1} \vartheta_r^\alpha\> & d_\infty\ge0
\end{array}\right.\]
\end{theorem}

\begin{example}
  If
  \begin{align*}
    \alpha &= (1,1,-1,2,-1,1;0)\\
    \beta &= (2,1,-1,1,0,2;-1)
  \end{align*}
  then
  \begin{align*}
    \TF_\alpha(R;\Z_p) &= A\<u_0 u_1 a_2 u_3^2 a_4 u_5\>\\
    \TF_{\alpha-1}(R;\Z_p) &= \frac A{\xi\phi(\xi),\phi^2(\xi)\phi^4(\xi)}\<\Sigma^{-1} \vartheta^\alpha_1\> \oplus \frac A{\phi^3(\xi)^2,\phi^4(\xi)}\<\Sigma^{-1} \vartheta^\alpha_3\>\\
    \TF_{\beta-1}(R;\Z_p) &= \frac A{\xi^2\phi(\xi)}\<\Sigma^{-1}\vartheta^\beta_1\> \oplus \frac A{\phi^3(\xi)\phi^5(\xi)^2}\<\Sigma^{-1}\vartheta^\beta_5\>
  \end{align*}
\end{example}

\begin{theorem}
\label{thm:rog-crys}
Let $A$ be a crystalline perfect prism, corresponding to a perfect $\F_p$-algebra $k=A/p$. The $\RU(\T)$-graded $\TF$ of $k$ is
\begin{align*}
  \TF_\rog(k) &= \BB[u_0, a_{i-1}^{-1} a_i \mid i\ge 1].
\end{align*}
The remaining $\RO(\T)$-graded groups are as follows; let $\alpha=(d_0,\dotsc,d_n;d_\infty)$.

If $d_\infty<0$, then $\TF_{\alpha-1}(k)$ has one summand for each collection of indices $0\le i_1\le\dotsc\le i_\ell\le n$ such that
\begin{itemize}
  \item each $d_{i_j}>0$;

  \item $\sum_{i_j < i < i_{j+1}} d_i = 0$ for all $j=1,\dotsc,\ell-1$;

  \item there is no $i_j < \hat\imath < i_{j+1}$ such that $\sum_{i_j < i < \hat\imath} d_i = 0$ or $\sum_{\hat\imath < i < i_{j+1}} d_i=0$;

  \item the collection is as long as possible with these properties.
\end{itemize}
Given such a subsequence, let
\begin{align*}
   o &= \sum_{j=1}^\ell d_{i_j}
\end{align*}
Then the corresponding summand is
  \[ A/p^o\<\Sigma^{-1} \vartheta_{i_\ell}^\alpha\>. \]

If $d_\infty\ge0$, we recursively (starting from $r=n$) define
\[
  e_r = \min\{d_r, s_r\},
  \qquad
  s_r = \sum_{i>r} (-e_i).
\]
$\TF_{\alpha-1}(k)$ then has one summand for each collection of indices $0\le i_1\le\dotsc\le i_\ell\le n$ such that
\begin{itemize}
  \item each $e_{i_j}>0$;

  \item $\sum_{i_j < i \le i_{j+1}} d_i = e_{i_{j+1}}$ for all $j=1,\dotsc,\ell-1$;

  \item there is no $i_j < \hat\imath < i_{j+1}$ such that $\sum_{i_j < i \le \hat\imath} d_i = e_{\hat\imath}$ or $\sum_{\hat\imath < i \le i_{j+1}} d_i=e_{i_{j+1}}$;

  \item the collection is as long as possible with these properties.
\end{itemize}
Given such a subsequence, let
\begin{align*}
   o &= \sum_{j=1}^\ell e_{i_j}
\end{align*}
Then the corresponding summand is
  \[ A/p^o\<\Sigma^{-1} \vartheta_{i_\ell}^\alpha\>. \]
\end{theorem}

\begin{example}
  If
  \begin{align*}
    \alpha &= (1,1,-1,2,-1,1;0)\\
    \beta &= (2,-1,2,-1,1;-1)\\
    \gamma &= (2,1,-1,1;-1)\\
    \intertext{then}
    e_\bullet(\alpha) &= (0,1,-1,1,-1,0)\\
    \intertext{and}
    \TF_\alpha(k) &= A\<u_0 (a_1^{-1} a_2) u_3 (a_3^{-1} a_4) u_5\>\\
    \TF_{\alpha-1}(k) &= A/p^2\<\Sigma^{-1}\vartheta^\alpha_3\>\\
    \TF_{\beta-1}(k) &= A/p^3\<\Sigma^{-1} \vartheta^\beta_4\> \oplus A/p^2\<\Sigma^{-1} \vartheta^\beta_2\>\\
    \TF_{\gamma-1}(k) &= A/p\<\Sigma^{-1}\vartheta^\gamma_3\> \oplus A/p^3\<\Sigma^{-1}\vartheta^\gamma_1\>.
  \end{align*}
\end{example}

In even degrees, we get an $\RU(\T)$-graded form of B\"okstedt periodicity; interestingly, $u_0=u_{\lambda_{-1}} u_{\lambda_0}^{-1}$ appears as slightly more fundamental than the B\"okstedt generator $u_{\lambda_{-1}}$. The classes $a_{i-1}^{-1} a_i$ reflect the collision of all the $\phi^i(\xi)$'s in the torsion case; note that $a_{i-1}^{-1} a_i = u_{i-1} u_i^{-1}$. 

\begin{remark}
We do not know the significance of the extremely intricate structure in the odd degrees, but we will share two leads. First, it may be possible to explain the odd groups in terms of the even ones via equivariant Anderson duality, as is done in \cite[\sec4.5]{Zengv1} or \cite[\sec6.2]{Zeng}. We have not taken this up as the homological algebra of $\m W[u_{\lambda_{-1}}]$-modules is considerably more complicated than that of $\m\Z$-modules. Second, in the crystalline case, torsion $A$-modules are parametrized by an affine Grassmannian, and the formulas of \cite[\sec\sec7,8]{WittAffGr} look somewhat similar to ours. For example, the sequences considered in \cite[Definition 7.1]{WittAffGr} correspond to actual $\T$-representations $V$ with $V^\T=0$, and the element called $\epsilon_j$ there is the degree of our class $a_{j-1}^{-1}a_j$.
\end{remark}

\begin{remark}
Angeltveit-Gerhardt computed not just $\TR^n_\rog(\F_p)$, but also $\TR^n_\rog(\Z;\F_p)$ and $\TR^n_\rog(\ell;V(1))$. This was applied in \cite{AGHtrunc} to calculate $\K_*(\Z[x]/(x^m),(x))$ up to extension, and in \cite{AGaxes} to compute $\K_*(\Z[x,y]/(xy),(x))$ up to extension. One could adapt the other two calculations from \cite{ROS1TR} in the same way as we have done for $\F_p$. Here is a different approach: extend our result to Breuil-Kisin prisms to determine $\TF_\rog(\O_K/\S_{W(k)}\psr z)$, and then use a spectral sequence as in \cite{KrauseNikolausBP} to compute $\TF_\rog(\O_K;\Z_p)$. This implies the computation of $\TF_\rog(\Z)$ or of other number fields. This might allow one to compute the above $\K$-groups completely. For $\ell$, one should consider $\TF_\rog(\ell/\S_p\psr{v_0,v_1})$.
\end{remark}

\subsection{Caveat lector}
\label{sub:caveat}
Tracking names of classes is not sufficient to identify the full multiplicative structure, since there is the possibility of exotic multiplications, as in \cite[Proposition 6.9]{Zeng}. While trying to work these out, we encountered a contradiction that we have not been able to resolve (it is probably due to the author not understanding hidden extensions well enough). For perfect $\F_p$-algebras, there is a subtle phenomenon that only starts once one gets out to $C_{p^5}$-representations, although the problem is more fundamental than that; we explain the issue in \sec\ref{sec:fail}. We have decided to release the paper in its current form since:
\begin{itemize}
\item the results are definitely correct, and already very interesting, for even degrees;

\item the paper is nearly unreadable, so the sooner people can start digesting it the better---were it not for this issue, the paper would have been released a year ago;

\item hopefully a reader will see what we are doing wrong.
\end{itemize}

\subsection{Notation}
All equations are meant up to units. Throughout $A$ denotes a perfect prism with Frobenius $\phi$ and generator $\xi$, and we write $R=A/\xi$. We set
\[ [p^n]_A \defeq \xi\phi(\xi)\dotsm\phi^{n-1}(\xi). \]
When we write $\T$ we really mean $C_{p^\infty}$. By $\THH(R;\Z_p)$ we mean the (proper-)genuine $\T$-spectrum. Given a $\T$-Mackey functor $\m M$, we write $\tr_n\m M$ for the sub-Mackey functor generated under transfers by $\res^\T_{C_{p^n}}\m M$. We also set $\Phi^n\m M=\m M/\tr_n\m M$. We write $\m W$ for the Mackey functor with $\m W(\T/C_{p^n})=A/[p^{n+1}]_A$, so that $\mpi_0\THH(R;\Z_p)=\m W$.

\subsection{Overview}
We begin in \sec\ref{sec:warmup} by explaining the calculation in detail for $C_p$-representations. This is intended to ease the reader into $\RO(\T)$-graded calculations, but also contains all the computational input needed for the HOTFSS in general. We construct the HOTFSS and HO\m{TR}SS in \sec\ref{sec:hotfss}, and write down the $E^1$ page. We analyze the HOTFSS for transversal perfect prisms in \sec\ref{sec:trans}, and for crystalline perfect prisms in \sec\ref{sec:crys}. In \sec\ref{sec:fail}, we explain the bug mentioned above.

\subsection{Acknowledgements}
We thank Mike Hill for the suggestion to study the slice filtration on $\THH$, which is what motivated this work, and especially for sharing \cite[Observation 2.32]{SulSliceTHH}, which led to the discovery of the $q$-gold relation. Dylan Wilson provided the crucial suggestion to use the gold elements. We are further grateful to Andrew Blumberg for supervision and guidance; to Aaron Royer for patiently fielding many questions about equivariant stable homotopy theory; to Teena Gerhardt for clarifying a confusion about the Tate spectral sequence; to Ben Antieau and Thomas Nikolaus for help sorting out a confusion about the $A$-algebra structure on $\TR^n$; and to Hood Chatham and Mingcong Zeng for helpful conversations. Additionally, this paper owes a tremendous intellectual debt to \cite{ROS1TR} and \cite{Zeng}.

Preliminary versions of the results in this paper appeared as part of the author's PhD thesis at the University of Texas at Austin.

% !TEX root=./rog.tex

\newcommand{\B}{\rule{0em}{1.2em}}

\newcommand{\gu}{u_0^{d_0} u_{\lambda_0}^{d_\infty}}
\newcommand{\ga}{a_0^{-d_0} u_{\lambda_0}^{d_\infty}}

\section{Warmup: \tops{$C_p$}{C\_p} representations}
\label{sec:warmup}

In this section we explain in detail how to compute $\TF_\rog(R;\Z_p)$ and $\mpi_\rog\THH(R;\Z_p)$ for $\rog$ ``a $C_p$-representation''. By this we mean that $\rog\in\{\alpha,\alpha-1\}$ for $\alpha\in\RU(\T)$ of the form $\alpha=(d_0;d_\infty)$. Explicitly, the irreducible decomposition of $\alpha$ is
\[ \alpha = (d_0-d_\infty)\lambda_0 + d_\infty\lambda_\infty. \]

We calculate $\TF_\rog(R;\Z_p)$ in \sec\ref{sub:warmup-hotfss}; the answer is the same in the crystalline as in the transversal case (this is not true for larger representations). We calculate $\mpi_\rog\THH(R;\Z_p)$ for $p$-torsionfree $R$ in \sec\ref{sub:warmup-hotrss-trans}, and for $p$-torsion $R$ in \sec\ref{sub:warmup-hotrss-tors}. %For entertainment's sake, we also show in \sec\ref{sub:warmup-spoke} how to incorporate the sign representation of $C_2$, as well as Hahn-Senger-Wilson's ``spoke'' generalization thereof \cite{Spoke}.

We will proceed via the isotropy separation square
\begin{equation}
\label{eq:isotropy}
\vcenter{\xymatrix{
  \Sigma T_h \ar@{=}[d] \ar[r] & T \ar[r]^-R \ar[d] & T^\Phi \ar[d]^-\varphi\\
  \Sigma T_h \ar[r] & T^h \ar[r] & T^t
}}
\end{equation}
where we write $T=\THH(R;\Z_p)$, $T_h=E\T_+\otimes T$, $T^\Phi=\widetilde{E\T}\tnsr T$, etc. In our notation, the $\Z$-graded homotopy of the right square is
\[\xymatrix{
  A[u_{\lambda_{-1}}] \ar[r] \ar[d] & A[t^{-1}] \ar[d]\\
  A[u_{\lambda_{-1}}, t] \ar[r] & A[t^{\pm}]
}\]
where $t=a_{\lambda_0}u_{\lambda_0}^{-1}$, subject to the relation $u_{\lambda_{-1}} t = \xi$.

Before doing any actual calculations, we explain the ``shape'' of what will happen. The bottom row of \eqref{eq:isotropy} is $u_{\lambda_0}$-periodic, meaning the groups here depend only on $d_0$. On the other hand, the right column is $a_{\lambda_0}$-periodic, meaning the groups here depend only on $d_\infty$. This is displayed in Figure~\ref{fig:shape-dim}: green, blue, and red denote ``orbit'', ``fixed-point'', and ``Tate'' information respectively. Under the motivic filtrations of \cite{BMS2}, these correspond to $\Nyg_{\le i}\prism$, $\Nyg^{\ge i}\prism$, and $\prism$. We also display this information using the usual grading convention in Figure~\ref{fig:shape-standard}; the double cone pattern seen here is familiar to equivariant homotopy theorists.

\begin{figure}[h]
  \begin{center}
    \begin{tikzpicture}[scale=1.5]
\draw[white] (-3, 0) -- (-3, 3) -- (8, 3) -- (8, 0) -- cycle;
% axes
\draw[thick] (-2, 1) -- (-2, 2) -- (-1, 2) -- (-1, 1) -- cycle;
\draw[anchor=north] (-1.5, 1) node {$d_\infty$};
\draw[anchor=east] (-2, 1.5) node {$d_0$};

% T_h
\draw[anchor=north] (0.5, 0) node {$\Sigma T_h$};
\fill[ForestGreen!40] (0, 0.5) -- (0, 1) -- (1, 1) -- (1, 0.5) -- cycle;
% \draw[ForestGreen,ultra thick] (0,1) -- (1, 0);

\draw[anchor=north] (0.5, 2) node {$\Sigma T_h$};
\fill[ForestGreen!40] (0, 2.5) -- (0, 3) -- (1, 3) -- (1, 2.5) -- cycle;
% \draw[ForestGreen,ultra thick] (0,3) -- (1, 2);

% T^h
\draw[anchor=north] (2.5, 0) node {$T^h$};
\fill[blue!40] (2, 0.5) -- (2, 1) -- (3, 1) -- (3, 0.5) -- cycle;
\fill[red!40] (2, 0) -- (2, 0.5) -- (3, 0.5) -- (3, 0) -- cycle;
% \draw[blue,ultra thick] (2,1) -- (3, 0);
% T^t
\draw[anchor=north] (4.5, 0) node {$T^t$};
\fill[red!40] (4, 0) -- (4, 1) -- (5, 1) -- (5, 0) -- cycle;

% T^\Phi
\draw[anchor=north] (4.5, 2) node {$T^\Phi$};
\fill[red!40] (4.5, 2) -- (4.5, 3) -- (5, 3) -- (5, 2) -- cycle;

% T
\draw[anchor=north] (2.5, 2) node {$T$};
\fill[ForestGreen!40] (2, 2.5) -- (2, 3) -- (2.5, 3) -- (2.5, 2.5) -- cycle;
% \draw[ForestGreen, ultra thick] (2, 3) -- (2.5, 2.5);
\fill[red!40] (2.5, 2) -- (2.5, 2.5) -- (3, 2.5) -- (3, 2) -- cycle;
\fill[blue!40] (2.5, 2.5) -- (2.5, 3) -- (3, 3) -- (3, 2.5) -- cycle;
% \draw[blue, ultra thick] (2.5, 3) -- (2.5, 2.5) -- (3, 2);

% box
\foreach \i in {0,1} {
  \foreach \j in {0,1,2} {
    \draw[thick] (2*\j, 2*\i) -- (2*\j+1, 2*\i) -- (2*\j+1, 2*\i+1) -- (2*\j, 2*\i+1) -- cycle;
  }
}
\end{tikzpicture}
  \end{center}
  \caption{Regions of $\pi_\rog$ in dimensional grading}
  \label{fig:shape-dim}
\end{figure}
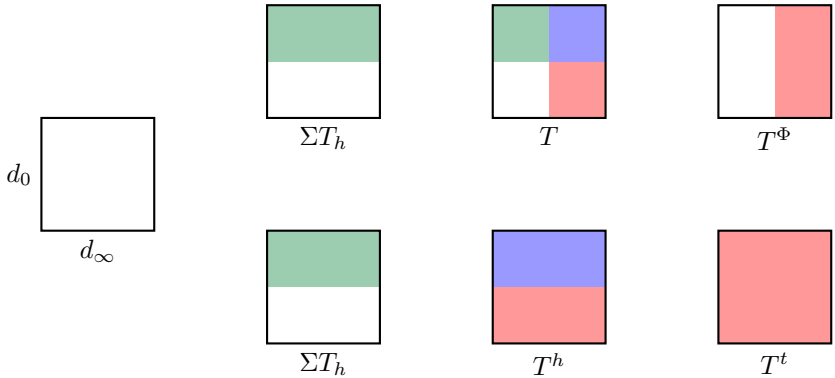

\begin{figure}[h]
  \begin{center}
    \begin{tikzpicture}[scale=1.5]
  \draw[white] (-3, 0) -- (-3, 3) -- (8, 3) -- (8, 0) -- cycle;
  % axes
  \draw[thick] (-2, 1) -- (-2, 2) -- (-1, 2) -- (-1, 1) -- cycle;
  \draw[anchor=north] (-1.5, 1) node {$d_\infty$};
  \draw[anchor=east] (-2, 1.5) node {$d_0-d_\infty$};

  % T_h
  \draw[anchor=north] (0.5, 0) node {$\Sigma T_h$};
  \fill[ForestGreen!40] (0,1) -- (1, 1) -- (1, 0) -- cycle;
  % \draw[ForestGreen,ultra thick] (0,1) -- (1, 0);

  \draw[anchor=north] (0.5, 2) node {$\Sigma T_h$};
  \fill[ForestGreen!40] (0,3) -- (1, 3) -- (1, 2) -- cycle;
  % \draw[ForestGreen,ultra thick] (0,3) -- (1, 2);

  % T^h
  \draw[anchor=north] (2.5, 0) node {$T^h$};
  \fill[blue!40] (2, 1) -- (3, 1) -- (3, 0) -- cycle;
  \fill[red!40] (2,0) -- (2, 1) -- (3, 0) -- cycle;
  % \draw[blue,ultra thick] (2,1) -- (3, 0);
  % T^t
  \draw[anchor=north] (4.5, 0) node {$T^t$};
  \fill[red!40] (4, 0) -- (4, 1) -- (5, 1) -- (5, 0) -- cycle;

  % T^\Phi
  \draw[anchor=north] (4.5, 2) node {$T^\Phi$};
  \fill[red!40] (4.5, 2) -- (4.5, 3) -- (5, 3) -- (5, 2) -- cycle;

  % T
  \draw[anchor=north] (2.5, 2) node {$T$};
  \fill[ForestGreen!40] (2, 3) -- (2.5, 3) -- (2.5, 2.5);
  % \draw[ForestGreen, ultra thick] (2, 3) -- (2.5, 2.5);
  \fill[red!40] (2.5, 2.5) -- (2.5, 2) -- (3, 2) -- cycle;
  \fill[blue!40] (2.5, 3) -- (2.5, 2.5) -- (3, 2) -- (3, 3) -- cycle;
  % \draw[blue, ultra thick] (2.5, 3) -- (2.5, 2.5) -- (3, 2);

  % box
  \foreach \i in {0,1} {
    \foreach \j in {0,1,2} {
      \draw[thick] (2*\j, 2*\i) -- (2*\j+1, 2*\i) -- (2*\j+1, 2*\i+1) -- (2*\j, 2*\i+1) -- cycle;
    }
  }
  \end{tikzpicture}
  \end{center}
  \caption{Regions of $\pi_\rog$ in standard grading}
  \label{fig:shape-standard}
\end{figure}
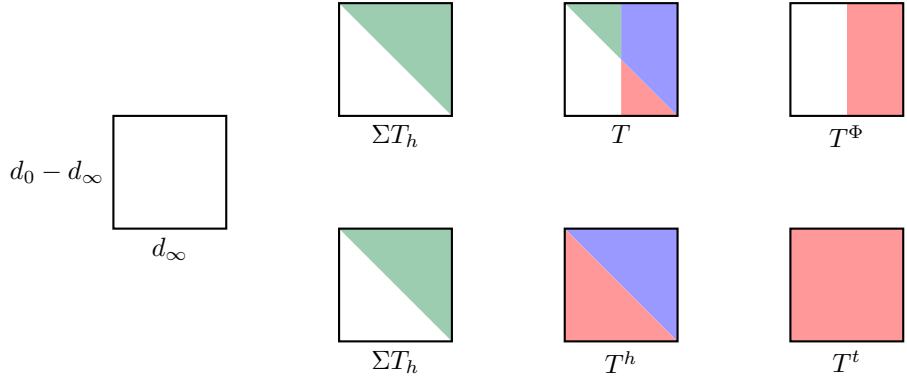

\subsection{Calculating \tops{$\TF_\rog$}{TF}}
\label{sub:warmup-hotfss}

By the above discussion, we have
\begin{align*}
  \TCmin_\alpha(R;\Z_p)
    &=
    \begin{cases}
      A\<u_{\lambda_{-1}}^{d_0} u_{\lambda_0}^{d_\infty-d_0}\> & d_0\ge0\\
      A\<t^{-d_0} u_{\lambda_0}^{d_\infty-d_0}\> & d_0<0
    \end{cases}
  & \TP_\alpha(R;\Z_p) &= A\<t^{-d_0} u_{\lambda_0}^{d_\infty-d_0}\>\\
    &=
    \begin{cases}
      A\<\gu\> & d_0\ge0\\
      A\<\ga\> & d_0<0
    \end{cases}
    &&= A\<\ga\>
\end{align*}
as well as $\TCmin_{\alpha-1}(R;\Z_p)=\TP_{\alpha-1}(R;\Z_p)=0$. It follows that
\[ \pi_\alpha^\T(\Sigma T_h) = 0,\quad \pi_{\alpha-1}^\T(\Sigma T_h) = A/\xi^{d_0}\<\Sigma^{-1}\ga\>. \]
We also have
\[
  \pi_\alpha^\T(T^\Phi) =
  \begin{cases}
    A\<\ga\> & d_\infty \ge 0\\
    0 & d_\infty < 0.
  \end{cases}
\]

If $d_0\ge0$ and $d_\infty<0$ (green region), we get
\begin{alignat*}{2}
  \TF_\alpha(R;\Z_p) &= 0, & \TF_{\alpha-1}(R;\Z_p) &= A/\xi^{d_0}\<\Sigma^{-1} \ga\>.\\
  \intertext{If $d_0<0$ and $d_\infty\ge0$ (red region), we get}
  \TF_\alpha(R;\Z_p) &= A\<\ga\>, & \qquad\TF_{\alpha-1}(R;\Z_p) &= 0.
  \intertext{If $d_0>0$ and $d_\infty\ge0$ (blue region), then the boundary map is
  \[\xymatrix{
    \pi^\T_\alpha(T^\Phi) \ar[r]^-\bdry \ar@{=}[d] & \pi^\T_{\alpha-1}(\Sigma T_h) \ar@{=}[d]\\
    A\<\ga\> \ar[r] & A/\xi^{d_0}\<\Sigma^{-1}\ga\>.
  }\]
Here there is a slight subtlety: the boundary map $\bdry$ is $\phi$-linear, so the kernel is $\phi^{-1}(\xi)^{d_0} \ga$. However, the map $R\colon\TF_\alpha\to\pi^\T_\alpha(T^\Phi)$ is $\phi^{-1}$-linear, so we end up with $\xi^{d_0}\ga=\gu$ as the generator. Because of this, we will pretend from now on that $R$ and $\bdry$ are both $A$-linear. In any event, for this case we have obtained}
\TF_\alpha(R;\Z_p) &= A\<\gu\>, &\TF_{\alpha-1}(R;\Z_p) &=0.
\end{alignat*}
% \begin{figure}[h]
%   \begin{center}
%     \begin{tabular}{c|cc||c|cc}
%     1 & $A/\xi^{d_0}\<\Sigma^{-1}\ga\>$ & & 1 &&\\
%     0 & & $A\<\ga\>$ & 0 & & $A\<\gu\>$ \B\\\hline
%     $E^1$ & $-1$ & $0$ & $E^\infty$ & $-1$ & $0$
%     \end{tabular}
%   \end{center}
%   \caption{HOTFSS with $d_0>0,\,d_\infty\ge0$}
% \end{figure}

\subsection{Calculating \tops{$\mpi_\rog$}{Mackey functors}: Transversal case}
\label{sub:warmup-hotrss-trans}

Now we do the Mackey functor version of this calculation when $A$ is transversal. On $\T/C_{p^n}$, we have
\begin{align*}
  \pi^{C_{p^n}}_\alpha(T^h)
    &=
    \begin{cases}
      A/[p^{n+1}]_A\<\gu\> & d_0\ge0\\
      A/\phi([p^n]_A)\<\ga\> & d_0<0
    \end{cases} &
  \pi^{C_{p^n}}_\alpha(T^t)
    &= A/\phi([p^n]_A)\<\ga\>.
\end{align*}
By transversality, the kernel of $A/[p^{n+1}]_A \xra{\xi^{d_0}} A/\phi([p^n]_A)$ is $\phi([p^n]_A)A/[p^{n+1}]_A\cong A/\xi$. This gives
\[ \pi^{C_{p^n}}_\alpha(\Sigma T_h) = A/\xi\<\phi([p^n]_A)\gu\>,\qquad \pi^{C_{p^n}}_{\alpha-1}(\Sigma T_h) = \frac A{\phi([p^n]_A),\xi^{d_0}}\<\Sigma^{-1} \ga \>. \]

Note that $\phi([p^n]_A)=\tr^{C_{p^n}}_e(1)$ is exactly the transfer from the trivial subgroup. Thus, the Mackey functor expression of the above calculation is
\[ \mpi_\alpha(\Sigma T_h) = \tr_0\m W\<\gu\>,\qquad \mpi_{\alpha-1}(\Sigma T_h) = \Phi^0\m W/\xi^{d_0}\<\Sigma^{-1}\ga\> \]
for $d_0\ge 0$, and $0$ when $d_0<0$. We also have
\[ \mpi_\alpha(T^\Phi) = \begin{cases}\Phi^0\m W\<\ga\> & d_\infty\ge0\\0 & d_\infty<0\end{cases} \]

As before, this determines the answer when $d_0\le0$ or $d_\infty<0$. To explain the case $d_0>0,\,d_\infty\ge0$, we will start using spectral sequence notation. Our notations for Mackey functors are listed in Figure~\ref{fig:mackey}; for reasons of space we have only displayed the Lewis diagram for $C_p$, but these really are $\T$-Mackey functors. With these notations, $\mpi_{\alpha(-1)}\THH(R;\Z_p)$ is calculated by the spectral sequence in Figure~\ref{fig:hotrss-warmup-trans}. The kernel of the differential is $\xi^{d_0}\ga=\gu$ by transversality.

\begin{figure}[h]
  \begin{center}
    \begin{tabular}{|r|c|c|c|c|}
    \firsthline
    \bfseries Symbol & $\MWitt$ & $\Mdual$ & $\Mtors$ & $\Mtors_i$\\\hline
    \bfseries Name & $\m W$ & $\tr_0 \m W$ & $\Phi^0\m W$ & $\Phi^0\m W/\xi^i$\\\hline
    \bfseries Lewis diagram & $\Lewis{A/\xi\phi(\xi)}{A/\xi}1{\phi(\xi)}$ & $\Lewis{A/\xi}{A/\xi}p1$ & $\Lewis{A/\phi(\xi)}0{}{}$ & $\Lewis{A/(\phi(\xi),\xi^{d_0})}0{}{}$\\\lasthline
    \end{tabular}
  \end{center}
  \caption{Hieroglyphics for Mackey functors}
  \label{fig:mackey}
\end{figure}

\begin{figure}[h]
  \begin{center}
    \begin{tabular}{c|cc||c|cc}
    1 & $\Mtors_{d_0}\<\Sigma^{-1}\ga\>$ & $\Mdual\<\gu\>$ & 1 && $\Mdual\<\gu\>$ \rdelim\}{2}{0em}[$\MWitt\<\gu\>$]\\
    0 & & $\Mtors\<\ga\>$ & 0 & & $\Mtors\<\gu\>$ \B\\\hline
    $E^1$ & $-1$ & $0$ & $E^\infty$ & $-1$ & $0$
    \end{tabular}
  \end{center}
  \caption{HO\m{TR}SS for transversal $A$ with $d_0>0,\,d_\infty\ge0$}
  \label{fig:hotrss-warmup-trans}
\end{figure}

We summarize the complete computation in Figure~\ref{fig:chart-trans}.

\begin{figure}[h]
  \begin{center}\includegraphics{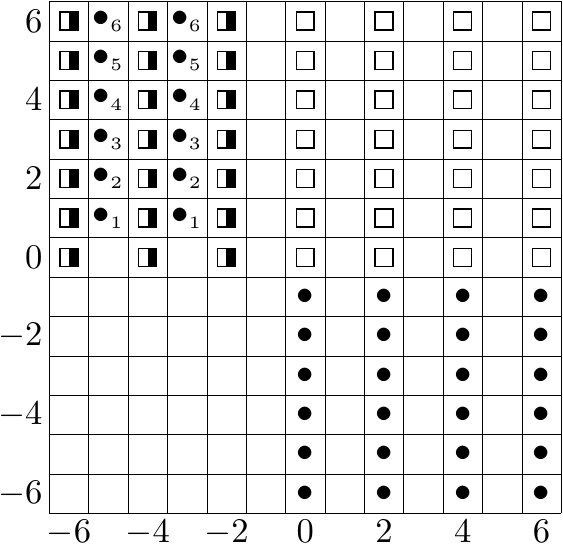}\end{center}
  \caption{$\mpi_{x+(y-\lf x/2\rf)\lambda}\THH(A/\xi;\Z_p)$ for transversal $A$}
  \label{fig:chart-trans}
\end{figure}

\subsection{Calculating \tops{$\mpi_\rog$}{Mackey functors}: Crystalline case}
\label{sub:warmup-hotrss-tors}

Now we do the Mackey functor calculation when $A$ is crystalline. On $\T/C_{p^n}$, we still have
\begin{align*}
  \pi^{C_{p^n}}_\alpha(T^h)
    &=
    \begin{cases}
      A/p^{n+1}\<\gu\> & d_0\ge0\\
      A/p^n\<\ga\> & d_0<0
    \end{cases} &
  \pi^{C_{p^n}}_\alpha(T^t)
    &= A/p^n\<\ga\>.
\end{align*}
However, now the kernel of $A/p^{n+1} \xra{p^{d_0}} A/p^n$ is $p^{n-\min(n,d_0)}A/p^{n+1}A$, while the cokernel is $A/p^{\min(n,d_0)}A$. We can express this in Mackey terms as
\[ \mpi_\alpha(\Sigma T_h) = \tr_{d_0}\m W\<\gu\>,\qquad \mpi_{\alpha-1}(\Sigma T_h) = \Phi^0\m W/p^{d_0}\<\Sigma^{-1}\ga\> \]
for $d_0\ge0$, and $0$ otherwise.

To explain the differential, we will again use spectral sequence notation. We need some more hieroglyphics for Mackey functors; let $\Mtors^i = \Phi^i\m W$ and $\Mdual_i=\tr_i\m W$ (beware that $\Mtors_i$ still denotes $\Phi^0\m W/p^i$). The spectral sequence is displayed in Figure~\ref{fig:hotrss-warmup-crys}.

\begin{figure}[h]
  \begin{center}
    \begin{tabular}{c|cc||c|cc}
    1 & $\Mtors_{d_0}\<\Sigma^{-1}\ga\>$ & $\Mdual_{d_0}\<\gu\>$ & 1 && $\Mdual_{d_0}\<\gu\>$ \rdelim\}{2}{0em}[$\MWitt\<\gu\>$]\\
    0 & & $\Mtors\<\ga\>$ & 0 & & $\Mtors^{d_0}\<\gu\>$ \B\\\hline
    $E^1$ & $-1$ & $0$ & $E^\infty$ & $-1$ & $0$
    \end{tabular}
  \end{center}
  \caption{HO\m{TR}SS for crystalline $A$ with $d_0>0,\,d_\infty\ge0$}
  \label{fig:hotrss-warmup-crys}
\end{figure}

We summarize the complete computation in Figure~\ref{fig:chart-crys}.

% \begin{figure}[h]
%   \begin{center}
%     \begin{tabular}{c|cc}
%     1 & & $\Mdual\<u_{\lambda_0}^{d_\infty}\>$ \rdelim\}{2}{0em}[$\MWitt\<u_{\lambda_0}^{d_\infty}\>$]\\
%     0 & & $\Mtors\<u_{\lambda_0}^{d_\infty}\>$ \B\\\hline
%     $E^1$ & $-1$ f

\begin{figure}[h]
  \begin{center}\includegraphics{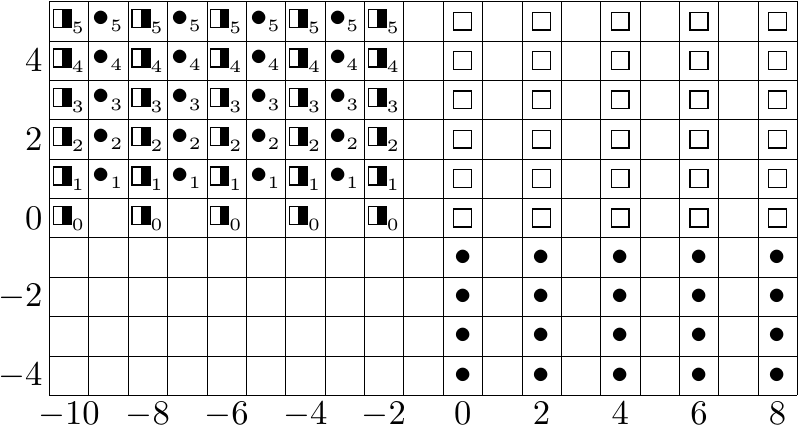}\end{center}
  \caption{$\mpi_{x+(y-\lf x/2\rf)\lambda}\THH(k)$}
  \label{fig:chart-crys}
\end{figure}

% \subsection{Sign and spoke representations}
% \label{sub:warmup-spoke}

% !TEX root=./rog.tex
\section{The general strategy}
\label{sec:hotfss}

Equivariantly, the fact that $\THH$ is cyclotomic is expressed by $\pi^\T_\rog(T^\Phi) = a_{\lambda_0}^\pm\TF_{\rog'}$. More precisely, classes in $\TF_{\rog'}$ correspond to classes in $\TPhi_\rog$ via $a_i\mapsto a_{i+1}$, $u_i\mapsto u_{i+1}$. Thus, in principle we can compute the result for $\alpha=(d_0,\dotsc,d_{n-1};d_\infty)$ by induction on $n$.

The difficulty with inductive arguments is that one needs to know the result in advance. In this section, we explain how to unroll this induction into a spectral sequence, adapted from the HOTRSS of \cite{ROS1TR}. In \sec\ref{sub:hotfss}, we introduce the homotopy orbits to $\TF$ spectral sequence (HOTFSS) calculating $\TF_\rog(R;\Z_p)$, and write down its $E^1$ page. We record the Mackey functor version of this spectral sequence, which we call the HO\m{TR}SS, in \sec\ref{sub:hotrss}.

\subsection{Homotopy orbits to \tops{$\TF$}{TF} spectral sequence}
\label{sub:hotfss}

Let $\cF_r$ denote the family of subgroups of $\T$ contained in $C_{p^r}$. There is a filtration
\[ E\cF_{0+} \to E\cF_{1+} \to \dotsb \to E\cF_{r+} \to S^0 \]
with associated graded
\[ E\cF_{0+},\, E\cF_{1+}\tnsr\widetilde{E\cF_0},\,\dotsc,\,E\cF_{(i+1)+}\tnsr\widetilde{E\cF_i},\,\dotsc,\,\widetilde{E\cF_r}. \]
Smashing with a $\T$-spectrum $X$, we get a spectral sequence calculating $X^\T$ in terms of $(X^{\Phi H})_{h\T/H}$. When $X=\THH$, we call this the homotopy orbits to TF spectral sequence (HOTFSS).

Let $\alpha=(d_0,\dotsc,d_{n-1};d_\infty)$. In the above construction, we will take $r=n-1$ (any $r\ge n-1$ would do). Since the homotopy orbits are even, it suffices to fix an $\alpha$ and consider only $\pi_\alpha$ and $\pi_{\alpha-1}$. Recall our notation
\[ \vartheta^\alpha_r = a_0^{-d_0}\dotsm a_r^{-d_r} u_{r+1}^{d_{r+1}}\dotsm u_{n-1}^{d_{n-1}} u_{\lambda_{n-1}}^{d_\infty}. \]

\begin{proposition}
The $E^1$ page of the homotopy orbits to $\TF$ spectral sequence is given by
\begin{align*}
  E^1_{n-r}
  &=
  \begin{cases}
    A/\phi^r(\xi)^{d_r}\<\Sigma^{-1} \vartheta^\alpha_r\> & d_r>0\\
    0 & d_r \le 0
  \end{cases}\\
  \intertext{for $0\le r<n$, and}
  E^1_0 &=
  \begin{cases}
    A\<\vartheta^\alpha_{n-1}\> & d_\infty\ge0\\
    0 & d_\infty < 0.
  \end{cases}
\end{align*}
\end{proposition}
\begin{proof}
The hard work was done in \sec\ref{sub:warmup-hotfss}, where we showed this is true for $n=1$. By cyclotomicity, the answer in the general case is isomorphic to this one up to a Frobenius twist. For the names, note that smashing with $\widetilde{E\mathcal F_r}$ inverts $a_0,\dotsc,a_{r-1}$.
\end{proof}

We see that there are three distinct regions of the HOTFSS we must understand:
\begin{enumerate}
\item $\TF_\alpha$, $d_\infty(\alpha)\ge0$;

\item $\TF_{\alpha-1}$, $d_\infty(\alpha)<0$;

\item $\TF_{\alpha-1}$, $d_\infty(\alpha)\ge0$.
\end{enumerate}

We list examples of the $E^1$ page below.

\begin{figure}[h]
  \begin{center}
    \begin{tabular}{c|cc}
      4 & $A/\xi^{2}\<a_{0}^{-2}u_{1}^{}u_{2}^{-2}u_{3}^{3}u_{\lambda_{3}}\>$\\
      3 & $A/\phi^{}(\xi)^{}\<a_{0}^{-2}a_{1}^{-1}u_{2}^{-2}u_{3}^{3}u_{\lambda_{3}}\>$\\
      2 & \\
      1 & $A/\phi^{3}(\xi)^{3}\<a_{0}^{-2}a_{1}^{-1}a_{2}^{2}a_{3}^{-3}u_{\lambda_{3}}\>$\\
      0 && $A\<a_{0}^{-2}a_{1}^{-1}a_{2}^{2}a_{3}^{-3}u_{\lambda_{3}}\>$\\\hline
      $E^1$ & $-1$ & 0
    \end{tabular}
  \end{center}
  \caption{HOTFSS for $\alpha=(2,1,-2,3;1)$}
\end{figure}

\begin{figure}[h]
  \begin{center}
    \begin{tabular}{c|cc}
      5 & $A/\xi^{3}\<a_{0}^{-3}u_{2}^{2}u_{3}^{}u_{4}^{-1}u_{\lambda_{4}}^{-1}\>$\\
      4 & \\
      3 & $A/\phi^{2}(\xi)^{2}\<a_{0}^{-3}a_{2}^{-2}u_{3}^{}u_{4}^{-1}u_{\lambda_{4}}^{-1}\>$\\
      2 & $A/\phi^{3}(\xi)^{}\<a_{0}^{-3}a_{2}^{-2}a_{3}^{-1}u_{4}^{-1}u_{\lambda_{4}}^{-1}\>$\\
      1 & \\
      0 & \\\hline
      $E^1$ & $-1$ & 0
    \end{tabular}
  \end{center}
  \caption{HOTFSS for $\alpha=(3,0,2,1,-1;-1)$}
\end{figure}

% \subsection{Extension problems}
% \label{sub:extns}

\subsection{Homotopy orbits to \tops{$\m{\TR}$}{TR} spectral sequence}
\label{sub:hotrss}

The filtration introduced above also gives rise to a spectral sequence of Mackey functors calculating $\mpi_\rog\THH(R;\Z_p)$. We call this the homotopy orbits to $\m\TR$ spectral sequence (HO\m{TR}SS). Once again, the hard work was done in \sec\ref{sub:warmup-hotrss-trans} and \sec\ref{sub:warmup-hotrss-tors}, so we can simply write down the $E^1$ page.

\begin{proposition}
  Let $A$ be a transversal perfect prism, and let $\alpha=(d_0,\dotsc,d_{n-1};d_\infty)$. The $E^1$ page of the homotopy orbits to $\m\TR$ spectral sequence is given by
  \begin{align*}
    E^1_{n-r}
    &=
    \begin{cases}
      \tr_r\Phi^{r-1}\m W\<\vartheta^\alpha_{r-1}\> \oplus \Phi^r\m W/\phi^r(\xi)^{d_r}\<\Sigma^{-1}\vartheta_r\> & d_r>0\\
      0 & d_r \le 0
    \end{cases}\\
    \intertext{for $0\le r<n$, and}
    E^1_0 &=
    \begin{cases}
      \Phi^{n-1}\m W\<\vartheta^\alpha_{n-1}\> & d_\infty\ge0\\
      0 & d_\infty < 0.
    \end{cases}
  \end{align*}
\end{proposition}

\begin{proposition}
  Let $A$ be a crystalline perfect prism, and let $\alpha=(d_0,\dotsc,d_{n-1};d_\infty)$. The $E^1$ page of the homotopy orbits to $\m\TR$ spectral sequence is given by
  \begin{align*}
    E^1_{n-r}
    &=
    \begin{cases}
      \tr_{r+d_r}\Phi^{r-1}\m W\<\vartheta^\alpha_{r-1}\> \oplus \Phi^r/p^{d_r}\m W\<\Sigma^{-1} \vartheta^\alpha_r\> & d_r>0\\
      0 & d_r \le 0
    \end{cases}\\
    \intertext{for $0\le r<n$, and}
    E^1_0 &=
    \begin{cases}
      \Phi^{n-1}\m W\<\vartheta^\alpha_{n-1}\> & d_\infty\ge0\\
      0 & d_\infty < 0.
    \end{cases}
  \end{align*}
\end{proposition}

We will not attempt to evaluate the HO\m{TR}SS in this paper, in order to find a closed form for $\mpi_\rog\THH(R;\Z_p)$. In the crystalline case this could certainly be done, but would be painful and not necessarily enlightening. In the transversal case, finding the kernels of the differentials is a non-trivial commutative algebra problem due to the gold relation.

% \begin{figure}[h]
%   \begin{center}
%     \begin{tabular}{c|cc|c|ccc}
%       2 & $\Phi^0\m W/\xi\<\Sigma^{-1} a_0^{-1} u_1^{-1}\>$ & $\tr_0\m W\<u_0 u_1^{-1}\>$ &
%       2 & $\Phi^1\m W/(\xi,\phi(\xi))\<\Sigma^{-1} a_0^{-1} u_1^{-1}\>$ & $\tr_0\m W\<u_0 u_1^{-1}\>$ &
%       \rdelim\}{2}{0em}[$\Phi^0\m W\<u_0 u_1^{-1}\>$]\\
%       1 & & $\tr_1 \Phi^0 \m W\<a_0^{-1} u_1^{-1}\>$ &
%       1 & & $\tr_1 \Phi^0 \m W\<u_0 u_1^{-1}\>$\\
%       0 && $\Phi^1\m W\<a_0^{-1} a_1\>$ &
%       0 && $\Phi^1\m W\<u_0 a_1\>$ \\\hline
%       $E^1$ & $-1$ & $0$ &
%       $E^\infty$ & $-1$ & $0$
%     \end{tabular}
%   \end{center}
%   \caption{{\HOTRSS} for $\alpha=(1,-1;0)$}
% \end{figure}

% !TEX root=./rog.tex
\section{\tops{$\TF_\rog$}{TF} for transversal perfect prisms}
\label{sec:trans}

In this section we analyze the HOTFSS for transversal perfect prisms. Let $\alpha=(d_0,\dotsc,d_{n-1};d_\infty)$. Note that for $s<r$, we have
\[ \vartheta_r^\alpha = \vartheta_s^\alpha \prod_{s<i\le r} \phi^i(\xi)^{-d_i}. \]

\begin{proof}[Proof of Theorem \ref{thm:rog-trans}]
\hfill
\begin{enumerate}
\item Region 1: $\TF_\alpha(R;\Z_p)$, $d_\infty(\alpha)\ge0$.

On the $E^1$ page, we start out with $A\<\vartheta^\alpha_{n-1}\>$. The target of the $d^{n-r}$ differential is $A/\phi^r(\xi)^{d_r}\<\Sigma^{-1} \vartheta^\alpha_r\>$ if $d_r>0$, and $0$ if $d_r\le0$. When $d_r>0$, the kernel of $d^r$ is $(\phi^r(\xi)^{d_r})$ by transversality, so the term $a_r^{-d_r}$ gets replaced with $u_r^{d_r}$ on the $E^{n-r+1}$ page.

\item Region 2: $\TF_{\alpha-1}(R;\Z_p)$, $d_\infty(\alpha)<0$.

There are no differentials, so all we need to do is determine the extensions. Let $s<r$ such that $d_s>0$ and $d_r>0$. In the HOTFSS, this will give

\begin{center}
\begin{tabular}{c|cc}
  $n-s$ & $A/\phi^s(\xi)^{d_s}\<\Sigma^{-1} \vartheta^\alpha_s\>$\\
  $\vdots$& $\vdots$\\
  $n-r$ & $A/\phi^r(\xi)^{d_r}\<\Sigma^{-1} \vartheta^\alpha_r\>$\\\hline
  $E^\infty$ & $-1$
\end{tabular}
\end{center}

If there were a nontrivial extension here, it would be
\[ A/\phi^s(\xi)^{d_s}\phi^r(\xi)^{d_r}\<\Sigma^{-1} \vartheta^\alpha_r\>. \]
In order for the names to match up, we must have $d_i=0$ for $s<i<r$.

\item Region 3: $\TF_{\alpha-1}(R;\Z_p)$, $d_\infty(\alpha)\ge0$.

First we determine the $E^\infty$ page. Suppose that $d_r>0$. By the discussion for Region (1), the $E^{n-r}$ page of the HOTFSS looks like

\begin{center}
\begin{tabular}{c|cc}
  $n-r$ & $A/\phi^r(\xi)^{d_r}\<\Sigma^{-1} a_0^{-d_0} \dotsm a_r^{-d_r} u_{r+1}^{d_{r+1}} \dotsm\>$\\
  $\vdots$ & $\vdots$\\
  0 & & $A\<a_0^{-d_0} \dotsm a_r^{-d_r} g_{r+1} \dotsm \>$\\\hline
  $E^{n-r}$ & $-1$ & $0$
\end{tabular}
\end{center}
where $g_i=u_i^{d_i}$ if $d_i\ge0$, and $g_i=a_i^{-d_i}$ if $d_i<0$.

We see that the image of $d^{n-r}$ is
\[
  f^\alpha_r \defeq\prod_{i>r,\,d_i<0} \phi^{i}(\xi)^{-d_i}
\]
so we get
\[ A/(\phi^r(\xi)^{d_r},f^\alpha_r)\<\Sigma^{-1} \vartheta^\alpha_r\> \]
on the $E^\infty$ page. The discussion of the extensions is then the same as in Region 2. \qedhere
\end{enumerate}
\end{proof}

We give several examples of the HOTFSS below. Extensions on the $E^\infty$ page are indicated with braces.

% \clearpage

\begin{figure}[h]
  \begin{center}
    \begin{tabular}{c|cc||c|cc}
2 & $A/\xi\<\Sigma^{-1} a_0^{-1} a_1\>$ & & 2 & $A/(\xi,\phi(\xi))\<\Sigma^{-1} a_0^{-1} a_1\>$\\
1 & & & 1 & \\
0 & & $A\<a_0^{-1} a_1\>$ & 0 & & $A\<u_0 a_1\>$\\\hline
$E^1$ & $-1$ & $0$ & $E^\infty$ & $-1$ & $0$
\end{tabular}

  \end{center}
  \caption{HOTFSS for $\alpha=(1,-1;0)$}
\end{figure}

\begin{figure}[h]
  \begin{center}
    \begin{tabular}{c|cc}
  6 & $A/\xi^{3}\<\Sigma^{-1}a_{0}^{-3}u_{1}^{}u_{2}^{-1}u_{3}^{}u_{5}^{2}u_{\lambda_5}^{-1}\>$ & \rdelim\}{2}{16em}[$\dfrac A{\xi^3\phi(\xi)}\<\Sigma^{-1}a_{0}^{-3}a_{1}^{-1}u_{2}^{-1}u_{3}^{}u_{5}^{2} u_{\lambda_5}^{-1}\>$] \\
  5 & $A/\phi^{}(\xi)^{}\<\Sigma^{-1}a_{0}^{-3}a_{1}^{-1}u_{2}^{-1}u_{3}^{}u_{5}^{2} u_{\lambda_5}^{-1}\>$ & \\
  4 &\\
  3 & $A/\phi^{3}(\xi)^{}\<\Sigma^{-1}a_{0}^{-3}a_{1}^{-1}a_{2}^{}a_{3}^{-1}u_{5}^{2}u_{\lambda_5}^{-1}\>$ & \rdelim\}{3}{16em}[$\dfrac A{\phi^3(\xi)\phi^5(\xi)^2}\<\Sigma^{-1}a_{0}^{-3}a_{1}^{-1}a_{2}^{}a_{3}^{-1}a_{5}^{-2}u_{\lambda_5}^{-1}\>$]\\
  2 & \\
  1 & $A/\phi^{5}(\xi)^{2}\<\Sigma^{-1}a_{0}^{-3}a_{1}^{-1}a_{2}^{}a_{3}^{-1}a_{5}^{-2}u_{\lambda_5}^{-1}\>$\\
  0 & \\\hline
  $E^\infty$ & $-1$ & $0$
\end{tabular}

  \end{center}
  \caption{HOTFSS for $\alpha=(3,1,-1,1,0,2;-1)$}
\end{figure}

\begin{figure}[h]
  \begin{center}
    \begin{tabular}{c|cc}
  4 & $A/\xi^{}\<\Sigma^{-1}a_{0}^{-1}u_{1}^{2}u_{2}^{-1}u_{3}^{3}\>$\\
  3 & $A/\phi^{}(\xi)^{2}\<\Sigma^{-1}a_{0}^{-1}a_{1}^{-2}u_{2}^{-1}u_{3}^{3}\>$\\
  2 & \\
  1 & $A/\phi^{3}(\xi)^{3}\<\Sigma^{-1}a_{0}^{-1}a_{1}^{-2}a_{2}^{}a_{3}^{-3}\>$\\
  0 & & $A\<a_{0}^{-1}a_{1}^{-2}a_{2}^{}a_{3}^{-3}\>$\\\hline
  $E^1$ & $-1$ & 0\\\hline\hline
  4 & $A/(\xi,\phi^2(\xi))\<\Sigma^{-1}a_{0}^{-1}u_{1}^{2}u_{2}^{-1}u_{3}^{3}\>$ & \rdelim\}{2}{15em}[$\dfrac A{\xi\phi(\xi)^2,\phi^2(\xi)}\<\Sigma^{-1}a_{0}^{-1}a_{1}^{-2}u_{2}^{-1}u_{3}^{3}\>$]\\
  3 & $A/(\phi(\xi)^2,\phi^2(\xi))\<\Sigma^{-1}a_{0}^{-1}a_{1}^{-2}u_{2}^{-1}u_{3}^{3}\>$\\
  2 & \\
  1 & \\
  0 & & $A\<u_0 u_1^2 a_{2} u_3^3\>$\\\hline
  $E^\infty$ & $-1$ & 0
\end{tabular}

  \end{center}
  \caption{HOTFSS for $\alpha=(1,2,-1,3;0)$}
\end{figure}

% !TEX root=./rog.tex
\section{\tops{$\TF_\rog$}{TF} for crystalline perfect prisms}
\label{sec:crys}

In this section we analyze the HOTFSS for crystalline perfect prisms. Let $\alpha=(d_0,\dotsc,d_{n-1};d_\infty)$.

The complication compared to the transversal case is that we now have the relation
\[ a_i^{-1} a_j = u_i u_j^{-1}. \]
This implies more generally that
\[ a_{i_1}^{-k_1}\dotsm a_{i_\ell}^{-k_\ell} = u_{i_1}^{k_1} \dotsm u_{i_\ell}^{k_\ell} \]
whenever $\sum k_j=0$. Moreover, for $s<r$ we have
\[ \vartheta_r^\alpha = p^{\sum_{s<i\le r} (-d_i)} \vartheta_s^\alpha. \]

\begin{proof}[Proof of Theorem \ref{thm:rog-crys}]
\hfill
\begin{enumerate}
\item Region 1: $\TF_\alpha(k)$, $d_\infty(\alpha)\ge0$.

  On the $E^1$ page, we start out with $A\<\vartheta^\alpha_{n-1}\>$. The target of the $d^{n-r}$ differential is $A/p^{d_r}\<\Sigma^{-1}\vartheta^\alpha_r\>$ if $d_r>0$, and $0$ if $d_r\le0$.

  Define $e_r, s_r$ inductively (starting from $r=n-1$) by
  \[
    e_r = \min\{d_r,s_r\},
    \qquad
    s_r = \sum_{i>r} (-e_i).
  \]
  Note that $s_r\ge0$ for all $r$, with equality if and only if $d_{r+1}\ge s_{r+1}$ (or $r=n-1$).

  Supposing $d_r>0$, we claim inductively that the image of $d^{n-r}$ is $(p^{s_r})$. The induction step is clear if $d_r<s_r$. Otherwise, the kernel of $d^{n-r}$ is $(p^{d_r-s_r})$. We can then rewrite
  \begin{align*}
    p^{d_r-s_r} a_r^{-d_r} a_{r+1}^{-d_{r+1}} a_{r+2}^{-d_{r+2}} \dotsm
    &= u_r^{d_r-s_r} (a_r^{-1} a_{r+1})^{-d_{r+1}} (a_r^{-1} a_{r+2})^{-d_{r+2}} \dotsm\\
    &= u_r^{d_r-s_r} (a_r^{-1} a_{r+1})^{s_r} (a_{r+1}^{-1} a_{r+2})^{s_{r+1}} \dotsm
  \end{align*}
  This completes the induction step, and also shows that the generator of $\TF_\alpha(k)$ has the desired form.

\item Region 2: $\TF_{\alpha-1}(k)$, $d_\infty(\alpha)<0$.

There are no differentials, so all we need to do is determine the extensions. Let $s<r$ such that $d_s>0$ and $d_r>0$. In the HOTFSS, this will give

\begin{center}
\begin{tabular}{c|cc}
  $n-s$ & $A/p^{d_s}\<\Sigma^{-1} \vartheta^\alpha_s\>$\\
  $\vdots$ & $\vdots$\\
  $n-r$ & $A/p^{d_r}\<\Sigma^{-1} \vartheta^\alpha_r\>$\\\hline
  $E^\infty$ & $-1$
\end{tabular}
\end{center}

If there were a nontrivial extension here, it would be
\[ A/p^{d_s+d_r}\<\Sigma^{-1} \vartheta^\alpha_r\>. \]
In order for the names to match up, we must have $\sum_{s<i<r} d_i=0$.

\item Region 3: $\TF_{\alpha-1}(k)$, $d_\infty(\alpha)\ge0$.

By the discussion for Region (1), the $E^\infty$ page of the HOTFSS looks like

\begin{center}
\begin{tabular}{c|cc}
  $n-s$ & $A/p^{e_s}\<\Sigma^{-1} \vartheta^\alpha_s\>$\\
  $\vdots$ & $\vdots$\\
  $n-r$ & $A/p^{e_r}\<\Sigma^{-1} \vartheta^\alpha_r\>$\\\hline
  $E^\infty$ & $-1$
\end{tabular}
\end{center}

If there were a nontrivial extension here, it would be
\[ A/p^{e_s+e_r}\<\Sigma^{-1} \vartheta^\alpha_r\>. \]
In order for the names to match up, we must have $e_r = \sum_{s<i\le r} d_i$. \qedhere
\end{enumerate}
\end{proof}

We give several examples of the HOTFSS below. Extensions on the $E^\infty$ page are indicated by lines.

\begin{figure}[h]
  \begin{center}
    \begin{tabular}{c|cc||c|cc}
2 & $A/p\<\Sigma^{-1} a_0^{-1} u_1^{-1}\>$ && 2 & $A/p\<\Sigma^{-1} a_0^{-1} u_1^{-1}\>$\\
1 && & 1 & \\
0 & & $A\<a_0^{-1} a_1\>$ & 0 & & $A\<a_0^{-1} a_1\>$\\\hline
$E^1$ & $-1$ & $0$ & $E^\infty$ & $-1$ & $0$
\end{tabular}

  \end{center}
  \caption{HOTFSS for $\alpha=(1,-1;0)$}
\end{figure}

\begin{figure}[h]
  \begin{center}
    \begin{tikzpicture}
\draw (0.3, 2.97) node[anchor=west] {$A/p\<\Sigma^{-1} a_{0}^{-1}u_{1}^{-1}u_{2}^{2}u_{3}^{-1}u_{4}^{}u_{\lambda_{4}}^{-1}\>$};
\draw (0.3, 1.97) node[anchor=west] {$A/p^2\<\Sigma^{-1} a_{0}^{-1}a_{1}^{}a_{2}^{-2}u_{3}^{-1}u_{4}^{}u_{\lambda_{4}}^{-1}\>$};
\draw (0.3, 0.97) node[anchor=west] {$A/p\<\Sigma^{-1} a_{0}^{-1}a_{1}^{}a_{2}^{-2}a_{3}^{}a_{4}^{-1}u_{\lambda_{4}}^{-1}\>$};
\draw (6, 2.95) node[anchor=west] {$A/p^2\<\Sigma^{-1} a_{0}^{-1}a_{1}^{}a_{2}^{-2}a_{3}^{}a_{4}^{-1}u_{\lambda_{4}}^{-1}\>$};
\draw (5, 3) -- (6, 3);
\draw (4.9, 1) -- (6, 3);
\foreach \x in {0,...,5}{
  \draw (0, {\x/2+0.5}) node {\x};
}
\draw (0, 0) node {$E^\infty$};
\draw (2.5, 0) node {$-1$};
\draw (6, 0) node {$0$};
\draw (0.3, -0.3) -- (0.3, 3.2);
\draw (-0.4, 0.25) -- (8, 0.25);
\end{tikzpicture}
  \end{center}
  \caption{HOTFSS for $\alpha=(1,-1,2,-1,1;-1)$}
\end{figure}

\begin{figure}[h]
  \begin{center}
    \begin{tikzpicture}
\draw (0.3, 6.5) node[anchor=west] {$A/p^{}\<\Sigma^{-1} a_{0}^{-1}u_{1}^{-1}u_{2}^{2}u_{3}^{-1}u_{4}^{}\>$};
\draw (0.3, 5.5) node[anchor=west] {$A/p^{2}\<\Sigma^{-1} a_{0}^{-1}a_{1}^{}a_{2}^{-2}u_{3}^{-1}u_{4}^{}\>$};
\draw (0.3, 4.5) node[anchor=west] {$A/p^{}\<\Sigma^{-1} a_{0}^{-1}a_{1}^{}a_{2}^{-2}a_{3}^{}a_{4}^{-1}\>$};
\draw (4.8, 4) node[anchor=west] {$A\<a_{0}^{-1}a_{1}^{}a_{2}^{-2}a_{3}^{}a_{4}^{-1}\>$};
\draw (0.3, 2.97) node[anchor=west] {$A/p^{}\<\Sigma^{-1} a_{0}^{-1}u_{1}^{-1}u_{2}^{2}u_{3}^{-1}u_{4}^{}\>$};
\draw (0.3, 1.97) node[anchor=west] {$A/p^{}\<\Sigma^{-1} a_{0}^{-1}a_{1}^{}a_{2}^{-2}u_{3}^{-1}u_{4}^{}\>$};
\draw (4.8, 0.5) node[anchor=west] {$A\<a_{0}^{-1}a_{1}^{}u_2 a_{2}^{-1}a_{3}^{}u_{4}\>$};
\draw (6, 2.97) node[anchor=west] {$A/p^2\<\Sigma^{-1} a_{0}^{-1}u_{1}^{-1}u_{2}^{2}u_{3}^{-1}u_{4}^{}\>$};
\draw (4.5, 2.97) -- (6, 2.97);
\draw (4.5, 1.97) -- (6, 2.97);
\foreach \x in {0,...,5}{
  \draw (0, {\x/2+4}) node {\x};
  \draw (0, {\x/2+0.5}) node {\x};
}
\draw (0, 3.5) node {$E^1$};
\draw (0, 0) node {$E^\infty$};
\draw (2.5, 0) node {$-1$};
\draw (6.4, 0) node {$0$};
\draw (2.5, 3.5) node {$-1$};
\draw (6.4, 3.5) node {$0$};
\draw (0.3, -0.3) -- (0.3, 6.8);
\draw (-0.4, 0.25) -- (8, 0.25);
\draw (-0.4, 3.2) -- (8, 3.2);
\draw (-0.4, 3.3) -- (8, 3.3);
\draw (-0.4, 3.7) -- (8, 3.7);
\end{tikzpicture}
  \end{center}
  \caption{HOTFSS for $\alpha=(1,-1,2,-1,1;0)$}
\end{figure}

\include{mackey}
% \include{rsss}
% !TEX root=./rog.tex

\section{Contradiction}
\label{sec:fail}

In this section we record the contradiction mentioned in \sec\ref{sub:caveat}, which we have not yet been able to resolve. Let $A$ be a crystalline perfect prism with quotient $k=A/p$. Note that for $s<r$, the gold relation implies
\[ \vartheta^\alpha_s = \vartheta^\alpha_r \cdot p^{\sum_{s<i\le r} d_i}, \]
which is bad because now things can have the same name.

Let $\alpha=(d_0,-1,2,-1,d_4;-1)$ with $d_0,d_4>0$, and let $\beta=(d_0,-1,2,-2,d_4;-1)$. The HOTFSS for $\alpha$ and for $\beta$ looks like:

\begin{center}
\vspace{.5em}
\begin{tabular}{c|cc}
  5 & $A/p^{d_0}\<\Sigma^{-1} \vartheta^\alpha_0\>$\\
  4 & \\
  3 & $A/p^2\<\Sigma^{-1} \vartheta^\alpha_2\>$\\
  2 & \\
  1 & $A/p^{d_4}\<\Sigma^{-1} \vartheta^\alpha_4\>$\\
  0 & \\\hline
  $E^\infty(\alpha)$ & $-1$ & 0
\end{tabular}
\qquad
\begin{tabular}{c|cc}
  5 & $A/p^{d_0}\<\Sigma^{-1} \vartheta^\beta_0\>$\\
  4 & \\
  3 & $A/p^2\<\Sigma^{-1} \vartheta^\beta_2\>$\\
  2 & \\
  1 & $A/p^{d_4}\<\Sigma^{-1} \vartheta^\beta_4\>$\\
  0 & \\\hline
  $E^\infty(\beta)$ & $-1$ & 0
\end{tabular}
\vspace{.5em}
\end{center}
which written out in full (other than the $u_{\lambda_4}^{-1}$, which doesn't really matter) is
\begin{center}
\vspace{.5em}
\begin{tabular}{c|cc}
  5 & $A/p^{d_0}\<\Sigma^{-1} a_0^{-d_0} u_1^{-1} u_2^2 u_3^{-1} u_4^{d_4}\>$\\
  4 & \\
  3 & $A/p^2\<\Sigma^{-1} a_0^{-d_0} a_1 a_2^{-2} u_3^{-1} u_4^{d_4}\>$\\
  2 & \\
  1 & $A/p^{d_4}\<\Sigma^{-1} a_0^{-d_0} a_1 a_2^{-2} a_3 a_4^{-d_4}\>$\\
  0 & \\\hline
  $E^\infty(\alpha)$ & $-1$ & 0
\end{tabular}
\qquad
\begin{tabular}{c|cc}
  5 & $A/p^{d_0}\<\Sigma^{-1} a_0^{-d_0} u_1^{-1} u_2^2 u_3^{-2} u_4^{d_4}\>$\\
  4 & \\
  3 & $A/p^2\<\Sigma^{-1} a_0^{-d_0} a_1 a_2^{-2} u_3^{-2} u_4^{d_4}\>$\\
  2 & \\
  1 & $A/p^{d_4}\<\Sigma^{-1} a_0^{-d_0} a_1 a_2^{-2} a_3^2 a_4^{-d_4}\>$\\
  0 & \\\hline
  $E^\infty(\beta)$ & $-1$ & 0
\end{tabular}
\vspace{1em}
\end{center}

Note that $u_1^{-1} u_2^2 u_3^{-1} = a_1 a_2^{-2} a_3$. So by tracking names, the extensions should be as follows:
\begin{align*}
  \TF_{\alpha-1}(k) &=
  \phantom\oplus A/p^2\<\Sigma^{-1} \vartheta^\alpha_2\>
  &
  \TF_{\beta-1}(k) &=
  \phantom\oplus A/p^{d_0}\<\Sigma^{-1} \vartheta^\beta_0\>\\
  &\phantom= \oplus A/p^{d_0+d_4}\<\Sigma^{-1} \vartheta^\alpha_4\> &
  &\phantom= \oplus A/p^2\<\Sigma^{-1} \vartheta^\beta_2\>\\
  &&&\phantom= \oplus A/p^{d_4}\<\Sigma^{-1} \vartheta^\beta_4\>
\end{align*}

We would like to determine the multiplication maps $\xymatrix@1{\TF_{\alpha-1}(k)\ar@<.5ex>[r]^-{a_3} & \ar@<.5ex>[l]^-{u_3} \TF_{\beta-1}(k)}$. Note that
\begin{itemize}
\item $\vartheta^\alpha_2 = \vartheta^\alpha_4$ if $d_4=1$\vspace{.5em}

\item $\vartheta^\beta_2 = \vartheta^\beta_4$ if $d_4=2$\vspace{.5em}
\end{itemize}
so in these cases there is the possibility of an exotic multiplication.

But there's a much bigger problem: basic linear algebra forbids having a $3\times 2$ matrix $A$ and a $2\times 3$ matrix $U$ such that
\[ AU=\begin{bmatrix}p&0&0\\0&p&0\\0&0&p\end{bmatrix},\qquad UA = \begin{bmatrix}p & 0\\0 &p\end{bmatrix}  \]
unless $d_0=1$ and/or $d_4=1$. So the gold relation seems to prohibit extensions with different numbers of summands if the summands are bigger than $A/p$. But such extensions do occur.

\bibliographystyle{amsalpha}
\bibliography{../bibliography}

\end{document}